\documentclass[12pt, reqno]{amsart}
\usepackage{amsmath,amssymb, amsthm}
\usepackage[colorlinks,urlcolor=red]{hyperref}
\usepackage{graphicx}
 \theoremstyle{plain}
\date{}

\theoremstyle{plain}
\newtheorem{thm}{Theorem}[section]

\newtheorem{lem}[thm]{Lemma}

\numberwithin{equation}{section}\theoremstyle{definition}
 \newtheorem{defn}[thm]{Definition}

\allowdisplaybreaks

\numberwithin{equation}{section}

\renewcommand{\to}{\longrightarrow}
\begin{document}
\title[relatively quasi-nonexpansive multivalued mappings ] {  strong convergence of a multi-step iterative process for relatively
 quasi-nonexpansive multivalued mappings and equilibrium problem in Banach spaces}
\title[convergence theorem for multivalued mappings and equilibrium problem ] {  Shrinking projection method for a sequence of relatively
 quasi-nonexpansive multivalued mappings
and equilibrium problem in Banach spaces}

\author[M. Eslamian]{ Mohammad Eslamian }
\date{}

\thanks{ }
\thanks{\it  {\rm  Young Researchers Club, Babol Branch, Islamic Azad
University, Babol, Iran.\\  Email: mhmdeslamian@gmail.com}}
\maketitle

\begin{abstract}
Strong convergence of a new iterative process based on  the Shrinking projection method to a common element of the set of common fixed points of an infinite family
of relatively  quasi-nonexpansive multivalued mappings
  and the solution set of an equilibrium problem in a Banach space is
established. Our results improved and extend the
corresponding results announced by many others.
\end{abstract}\maketitle\noindent {\bf Key words}: Shrinking projection method  , equilibrium problem,  common fixed
point,  relatively  quasi-nonexpansive multivalued mappings.\\
\noindent {\bf 2000 Mathematics Subject Classification}:
47H10,47H09.
\section{Introduction}

A nonempty subset $C$ of a Banach space $E$ is called proximinal if for
each $x\in E$, there exists an element $y\in C$ such that
$$\parallel x-y\parallel=dist(x,C)=inf\{\parallel x-z\parallel: z\in C\}.$$
 We denote by $ N(C),CB(C)$ and
 $P(C)$ the collection of all nonempty subsets, nonempty closed bounded subsets and nonempty
 proximinal bounded subsets of $C,$ respectively. The
Hausdorff metric $H$ on $CB(C)$ is defined by
$$H(A,B):=\max \{\sup_{x\in A}dist(x,B),\sup_{y\in B}dist(y,A)\},$$ for
all $A,B\in CB(C).$\\ Let $T:E\to N(E)$ be a multivalued mapping. An
element $x\in E$ is said to be a fixed point of $T$, if  $x\in Tx$.
The set of fixed points of $T$ will be denoted by $F(T)$.
\begin{defn}
A multivalued mapping $T:E\to CB(E)$ is called
\begin{enumerate}\item[ (i)] nonexpansive if
$$H(Tx,Ty)\le \|x-y\|,\quad x,y\in E.$$\item[ (ii)] quasi-nonexpansive if
$$F(T)\neq \emptyset \qquad and \qquad H(Tx,Tp)\le\|x-p\|, \quad  x\in E,\quad
p\in F(T) .$$
\end{enumerate}\end{defn}
The theory of multivalued
mappings has applications in control theory, convex optimization,
differential equations and economics. Theory of nonexpansive multivalued
 mappings is harder than the corresponding theory of nonexpansive
single valued mappings. Different iterative processes
have been used to approximate fixed points of multivalued
nonexpansive mappings (see [1-7]).

Let $E$ be a real Banach space and let $E^{*}$ be the dual space
of $E$. Let  $C$ be a closed convex subset of  $E$.  Let $F$ be a
bifunction from $C\times C$ into $\mathbb{R}$, where $\mathbb{R}$
is the set of real numbers. The equilibrium problem for $F:C\times
C \to\mathbb{R}$ is to find $\widehat{x}\in C$ such that
$$F(\widehat{x},y)\geq 0, \qquad \forall y\in C.$$ The set of solutions  is
denoted by  $EP(F)$.
 Equilibrium problems,  have had a great
impact and influence in the development of several branches of pure and applied
sciences.  Numerous problems in physics, optimization and economics reduce
to finding a solution of the  equilibrium problem . Some methods have been proposed to
solve the equilibrium problem in a Hilbert space. See [8-10].\par
Let $E$ be a real Banach space with norm $\|.\|$  and let $J$ be
the normalized duality mapping from $E$ into $2^{E^{*}}$ given by
$$Jx=\{x^{*}\in  E^{*}: \langle x,x^{*}\rangle=\|x\|\|x^{*}\|,\|x\|=\|x^{*}\|\}$$  for all $x \in
E$, where $E^{*}$ denotes the dual space of $E$ and
$\langle.,.\rangle$ the generalized duality pairing between $E$
and $E^{*}$.
 As we all know that if $C$ is a nonempty closed convex subset
of a Hilbert space $H$ and $P_{C}:H \to C$ is the metric
projection of $H$ onto $C$, then $P_{C}$ is nonexpansive. This
fact actually characterizes Hilbert spaces and consequently, it
is not available in more general Banach spaces. In this
connection, Alber \cite{alber} recently introduced a generalized
projection operator $E^{*}$  in a smooth  Banach space $E$ which is an
analogue of the metric projection in Hilbert spaces. Consider the
functional defined by $$\phi(x,y)=\|x\|^{2}-2\langle
x,Jy\rangle+\|y\|^{2},\qquad x,y\in E.$$ Observe that, in a
Hilbert space H, $\phi(x,y)$ reduces to $\|x-y\|^{2}.$
The generalized projection $\Pi_{C}:E\to C$  is a map that
assigns to an arbitrary point $x\in E$ the minimum point of the
functional $\phi(x,y)$ that is, $\Pi_{C}x=\overline{x}$, where
$\overline{x}$ is the solution to the minimization problem
$$\phi(\overline{x},x)=\inf_{y\in C}\phi(y,x)$$  The existence and uniqueness of the
operator $\Pi_{C}$  follows from the properties of the functional
$\phi(x,y)$ and strict monotonicity of the mapping $J$ (see, for
example, \cite{alber, kt,cio}). In Hilbert spaces,
$\Pi_{C}=P_{C}$. It is obvious from the definition of function
$\phi$ that
\begin{equation}(\|y\|-\|x\|)^{2}\leq \phi(x,y)\leq (\|y\|+\|x\|)^{2}\qquad
\forall x,y\in E.\end{equation} \maketitle\noindent {\bf{Remark
1:}} If $E$ is a reflexive, strictly convex and smooth Banach
space, then for $x,y\in E$, $\phi(x,y)=0$ if and only if $x=y$
(see \cite{cio, tak1}).

Let $C$ be a  nonempty closed convex subset of a smooth, strictly
convex and reflexive Banach space $E$, and let $T$ be a mapping from
$C$ into itself. We denote by $F(T)$ the set of fixed points of $T$.
A point $p\in C$ is said to be an asymptotic \cite{brz} fixed point
of $T$, if $C$ contains a sequence $\{x_{n}\}$ which converges
weakly to $p$ such that $\lim_{n\to\infty}\|x_{n}-Tx_{n}\|=0$. The
set of asymptotic fixed points of $T$ will be denoted by
$\widetilde{ F(T)}$. A mapping $T$ is said to be relatively
nonexpansive \cite{mt0,mt00}, if $\widetilde{ F(T)} =F(T)$ and $\phi
(p,Tx)\leq \phi(p,x)$ for all $x\in C$ and $p\in F(T)$. $T$ is said
to be relatively quasi-nonexpansive (\cite{rr1,rr2}) if $F(T)\neq
\emptyset$ and $\phi (p,Tx)\leq \phi(p,x)$ for all $x\in C$ and
$p\in F(T)$. The class of relatively quasi-nonexpansive mappings is
bigger than the class of relatively nonexpansive mappings which
requires the strong restriction: $\widetilde{ F(T)} =F(T).$

In the recent years, approximation of fixed points of relatively quasi-nonexpansive
 mappings by iteration has been studied by many
authors, see [16-23].

Very recently, Eslamian and Abkar \cite{ea} introduce the relatively  quasi-nonexpansive multivalued mapping as follows:
  \begin{defn} Let $C$ be a closed convex subset of a smooth
Banach space $E$, and $T:C\to N(C)$ be a multivalued mapping. We set
$$\Phi(Tx,Tp)=max\{\sup_{q\in Tp}\inf_{y\in Tx}\phi(y,q),\sup_{y\in Tx}\inf_{q\in
Tp} \phi(y,q)\}.$$ We call $T$ is relatively quasi-nonexpansive
multivalued mapping if $F(T)\neq \emptyset$ and
$$\Phi(Tx,Tp)\leq \phi(x,p),\qquad \forall p\in F(T),\quad
\forall x\in C.$$
\end{defn}
\noindent {\bf Remark }: In a Hilbert space,
 $\Phi(Tx,Ty)=H(Tx,Ty)^{2}$, and hence relatively quasi-nonexpansivness is
 equivalent to quasi-nonexpansivness.\par
In \cite{ea} the author  presented some example of relatively quasi-nonexpansive multivalued mapping.
Another generalization of relatively nonexpansive mapping to multivalued mappings presented by Simin Homaeipour \cite{h} as follows.
\begin{defn} Let $C$ be
a closed convex subset of a smooth Banach space $E$, and $T:C\to
N(C)$ be a multivalued mapping, $T$ is called relatively
nonexpansive mapping if the following conditions are satisfied:\\
(i) $F(T)\neq \emptyset$,\\(ii) $\phi (p,z)\leq \phi(p,x)$ for
all $z\in Tx$, $x\in C$  and $p\in F(T)$,\\(iii) $\widetilde{
F(T)} =F(T).$
\end{defn}
In a Hilbert space $H$, condition (ii) is equivalent to
$$\|p-z\|\leq\|p-x\|,\qquad \forall z\in Tx,\quad  \forall x\in C, \quad \forall p\in F(T).$$
Now if put $H=\mathbb{R}$ and $Tx=[0,x]$, we observe that $T$ is a nonexpansive multivalued mapping but $T$ is not 
relatively nonexpansive  mapping. Hence in spite of single valued case, relatively nonexpansive multivalued mapping is not equivalent to quasi-nonexpansive multivalued mapping.

\begin{defn} A multivalued mapping  $T$ is called closed if $x_{n}\to w$ and
$\lim_{n\to\infty}dist(x_{n},Tx_{n})=0$, then $w\in T(w).$
\end{defn}
In \cite{ea}, Eslamian and Abkar proved the following theorem.
\begin{thm}
Let $E$ be a uniformly smooth and uniformly convex Banach space, and
let $C$ be a nonempty closed convex subset of $E$. Let $F$ be a
bifunction from $C\times C$ into $\mathbb{R}$ satisfying
$(A1)-(A4).$ Let $T_{i}:C\to P(C)$, $ i=1,2,...,m$,  be a finite
family of multivalued mappings such that $P_{T_{i}}$ is closed and
relatively quasi-nonexpansive. Assume that $\mathcal
{F}=\bigcap_{i=1}^{m}F(T_{i})\bigcap EP(f)\neq\emptyset.$ For
$x_{0}\in C$ and $C_{0}=C$, let $\{x_{n}\}$ be a sequences generated
by  the following algorithm:
$$\begin {cases}
y_{n,1}=J^{-1}((1-a_{n,1})Jx_{n}+a_{n,1}Jz_{n,1}),\\
y_{n,2}=J^{-1}((1-a_{n,2})Jx_{n}+a_{n,2}Jz_{n,2}),\\...\\
y_{n,m}=J^{-1}((1-a_{n,m})Jx_{n}+a_{n,m}Jz_{n,m}),
\\u_{n}\in C\, such\, that\,
f(u_{n},y)+\frac{1}{r_{n}}\langle
y-u_{n},Ju_{n}-Jy_{n,m}\rangle\geq0 ,\qquad \forall y\in C,\\
C_{n+1}=\{z\in C_{n}: \phi(z,u_{n})\leq \phi(z,x_{n})\},\\
x_{n+1}=\prod_{C_{n+1}}x_{0}, \qquad \forall n\geq 0\end{cases}
$$ where   $z_{n,1}\in P_{T_{1}}x_{n}$ and  $z_{n,i}\in P_{T_{i}}y_{n,i-1}$ for $i=2,...,m$ and  $J$ is
 the duality mapping on $E$. Assume that $\sum_{i=1}^{m} a_{n,i}=1$, $\{a_{n,i}\}\in [a,b]\subset (0,1)$
 and $\{r_{n}\}\subset[c,\infty)$ for some $c>0$. Suppose that  $P_{T_{i}}$ is uniformly
 continuous with respect to the Hausdorff metric for $i=2,3,...,m$.
 Then $\{x_{n}\}$ converges strongly
 to $\Pi _{\mathcal{F}}x_{0}.$
\end{thm}
In this paper, we introduce a new shirking  projection algorithm for finding
a common element of the set of common fixed points of an infinite
family of relatively quasi-nonexpansive multivalued mappings and
the set of solutions of an equilibrium problem in uniformly smooth and uniformly convex Banach spaces.
Strong convergence to common elements of two set is established. Our results improved and extend the corresponding
results announced by many others.
 \section{Preliminaries}
  A Banach space $E$ is said to be strictly convex if
$\|\frac{x+y}{2}\|<1$ for all $x,y\in E$ with $\|x\|=\|y\|=1$ and
$x\neq y$ . It said to be uniformly convex if
$\lim_{n\to\infty}\|x_{n}-y_{n}\|=0$ for any two sequences
$\{x_{n}\}$ and $\{y_{n}\}$ in $E$ such that
$\|x_{n}\|=\|y_{n}\|=1$ and
$\lim_{n\to\infty}\|\frac{x_{n}+y_{n}}{2}\|=1$. Let $U=\{x\in
E:\|x\|=1\}$ be the unit sphere of $E$. Then the Banach space $E$
is said to be smooth provided $$\lim_{t\to
0}\frac{\|x+ty\|-\|x\|}{t}$$ exists for each $x,y\in U.$ It is
also said to be uniformly smooth if the limit is attained
uniformly for $x,y\in E.$ It is well known that if $E^{*}$ is
uniformly convex, then $J$ is uniformly continuous on bounded
subsets of $E$.

\begin{lem}(\cite{kt})\label{1}
Let $E$ be a uniformly convex and smooth Banach space and let
$\{x_{n}\}$ and $\{y_{n}\}$ be two sequences in $E.$ If
$\phi(x_{n},y_{n})\to 0$ and either $\{x_{n}\}$ or $\{y_{n}\}$ is
bounded, then $x_{n}-y_{n}\to 0.$
\end{lem}

\begin{lem}(\cite{alber})\label{2}
Let $C$ be a  nonempty closed convex subset of a  smooth Banach
space $E$ and  $x\in E$. Then $x_{0}=\Pi_{C}x$ if and only if
$$\langle x_{0}-y,Jx-Jx_{0}\rangle\geq 0,\qquad \forall y\in C.$$
\end{lem}
\begin{lem}(\cite{alber})\label{3}
Let $E$ be a reflexive, strictly convex and smooth Banach space. Let
$C$ be a nonempty closed convex subset of $E$ and let $x\in E$. Then
$$\phi(y,\Pi_{C}x)+\phi(\Pi_{C}x,x)\leq \phi(y,x),\qquad \forall y\in
C.$$
\end{lem}

\begin{lem}( \cite{ckw})\label{4}
Let $E$ be a uniformly convex Banach
space and let $B_{r}(0)=\{x\in E: \parallel x\parallel\leq r\}$,
for $r>0$.
 Then, for any given
sequence $\{x_{n}\}_{n=1}^{\infty}\subset B_{r}(0)$
 and for any given sequence $\{a_{n}\}_{n=1}^{\infty}$ of positive
numbers with $\sum_{n=1}^{\infty}a_{n}=1$
there exists a continuous, strictly increasing and
convex function $g:[0,\infty)\to [0,\infty)$ with
$g(0)=0$ such that  that for any positive
integers $i, j$ with $i < j$,
$$\|\sum_{n=1}^{\infty}a_{n}x_{n}\|^{2}\leq \sum _{n=1}^{\infty}a_{n}\|x_{n}\|^{2}-a_{i}a_{j} g(\|x_{i}-x_{j}\|.$$
\end{lem}
For solving the equilibrium problem, we assume that the
bifunction  $F$ satisfies the following conditions:
\begin{enumerate}\item[ (A1)] $F(x,x)=0$ for all $x\in C$,
\item[ (A2)]$F$ is monotone, i.e. $F(x,y)+F(y,x)\leq 0$
for any $x,y\in C,$
\item[ (A3)]$F$ is upper-hemicontinuous, i.e. for each
$x,y,z\in C$,
$$ \limsup _{t\to 0^{+}} F(tz+(1-t)x,y)\leq F(x,y)$$
\item[(A4)] $F(x,.)$ is convex and lower semicontinuous for each $x\in C.$
\end{enumerate}
The following lemma was proved in \cite{r1}.
\begin{lem}\label{5}
Let $C$ be a nonempty closed convex subset of  a smooth, strictly
convex, and reflexive Banach space $E$ ,  let $F$ be a bifunction
of $C\times C$ into $\mathbb{R}$ satisfying $(A1)-(A4).$ Let
$r>0$ and $x\in E$. Then, there exists $z\in C$ such that
$$F(z,y)+\frac{1}{r}\langle y-z,Jz-Jx\rangle\geq0 \qquad \forall
y\in C.$$
\end{lem}
The following lemma was given in \cite{mt2}.
\begin{lem}\label{6}
Let $C$ be a nonempty closed convex subset of  a smooth, strictly
convex, and reflexive Banach space $E$ ,  let $F$ be a bifunction
of $C\times C$ into $\mathbb{R}$ satisfying $(A1)-(A4).$ Let
$r>0$ and $x\in E$.
  define a mapping  $T_{r}:E\to C$ as follows: $$S_{r}x=\{z\in C :  F(z,y)+\frac{1}{r}\langle
y-z,Jz-Jx\rangle\geq0 ,\forall y\in C\}.$$ Then, the following
hold:
\begin{enumerate}\item[ (i)] $S_{r}$ is single valued;
\item[ (ii)]$S_{r}$ is  firmly nonexpansive-type mapping, i.e., for any $x,y\in E$,
$$\langle S_{r}x-S_{r}y,JS_{r}x-JS_{r}y\rangle\leq\langle S_{r}x-S_{r}y,Jx-Jy\rangle;$$
\item[ (iii)]$F(S_{r})=EP(F);$
\item[ (iv)] $EP(F)$ is closed and convex.
\end{enumerate}
\end{lem}
\begin{lem}(\cite{mt2})\label{7}
Let $C$ be a nonempty closed convex subset of  a smooth, strictly
convex, and reflexive Banach space $E$,  let $F$ be a bifunction
of $C\times C$ into $\mathbb{R}$ satisfying $(A1)-(A4),$  and let
 $r>0.$ Then for all $x\in E$ and $q\in
 F(S_{r})$,$$\phi(q,S_{r}x)+\phi(S_{r}x,x)\leq \phi(q,x).$$
\end{lem}

\begin{lem} \cite{ea} \label{8} Let $C$ be a nonempty closed convex subset of a uniformly convex and smooth Banach
space $E$. Suppose $T:C\to P(C)$ is a multivalued mapping such
that $P_{T}$ is a relatively quasi-nonexpansive multivalued
mapping where $$P_{T}(x)=\{y\in Tx:\|x-y\|=dist(x,Tx)\}.$$ If
$F(T)\neq\emptyset$, then $F(T)$ is closed and convex.
\end{lem}
\section{Main Result}
In this section, we prove strong convergence theorems for finding
a common element of the set of solutions for an equilibrium
problem and the set of fixed points of an infinite family of
relatively quasi- nonexpansive multivalued mappings in a Banach
space.
\begin{thm}
Let $E$ be a uniformly smooth and uniformly convex Banach space,
and let $C$ be a nonempty closed convex subset of $E$. Let $F$ be
a bifunction from $C\times C$ into $\mathbb{R}$ satisfying
$(A1)-(A4).$ Let $T_{i}:C\to N(C)$,  be a sequence of multivalued mappings such that for each $i\in\mathbb{N}$, $P_{T_{i}}$ is
 closed relatively quasi- nonexpansive multivalued
mappings and $\mathcal {F}=\bigcap_{i=1}^{\infty}F(T_{i})\bigcap
EP(F)\neq\emptyset$.  For $x_{0}\in C$ and $C_{0}=C$, let
$\{x_{n}\}$ be a sequence generated by  the following algorithm:
$$\begin {cases}
y_{n}=J^{-1}(a_{n,0}Jx_{n}+\sum_{i=1}^{\infty}a_{n,i}Jz_{n,i}),
\\u_{n}\in C :
F(u_{n},y)+\frac{1}{r_{n}}\langle
y-u_{n},Ju_{n}-Jy_{n}\rangle\geq0 ,\qquad \forall y\in C,\\
C_{n+1}=\{z\in C_{n}: \phi(z,u_{n})\leq \phi(z,x_{n})\},\\
x_{n+1}=\prod_{C_{n+1}}x_{0}, \qquad \forall n\geq 0,\end{cases}$$
 where  $\sum_{i=0}^{\infty} a_{n,i}=1$ and   $z_{n,i}\in P_{T_{i}}x_{n}$. Assume further that  $\liminf_{n}a_{n,0}\,a_{n,i}>0$, $\{r_{n}\}\subset
 (0,\infty)$ and $\liminf_{n}r_{n}>0$. Then $\{x_{n}\}$ converges strongly
 to $\Pi _{\mathcal{F}}x_{0}$, where $\Pi_{\mathcal{F}}$ is the
 projection of $E$ onto $\mathcal{F}$.
\end{thm}
\begin{proof}
First, we show by induction that $\mathcal
{F}=(\bigcap_{i=1}^{\infty}F(T_{i}))\bigcap EP(F)\subset C_{n}$ for all
$n\geq 0$. From $C_{0}=C$, we have $\mathcal{F}\subset C_{0}.$ We
suppose that $\mathcal{F}\subset C_{n}$ for some $n\geq 0$. Let
$u\in \mathcal{F}$, then we have $P_{T_{i}}u=\{u\}, (i\in\mathbb{N})$.
 Since $S_{r_{n}}$ and $T_{i}$ are
relatively quasi- nonexpansive, we have
\begin{multline*}
\phi(u,u_{n})=\phi(u,S_{r_{n}}y_{n})\leq\phi(u,y_{n})=\phi(u,J^{-1}(a_{n,0}Jx_{n}+\sum_{i=1}^{\infty}a_{n,i}Jz_{n,i}))\\
=\|u\|^{2}-2\langle
u,a_{n,0}Jx_{n}+\sum_{i=1}^{\infty}a_{n,i}Jz_{n,i}\rangle+
\|a_{n,0}Jx_{n}+\sum_{i=1}^{\infty}a_{n,i}Jz_{n,i}\|^{2}\\
\leq \|u\|^{2}-2a_{n,0}\langle u,Jx_{n}\rangle -2\sum_{i=1}^{\infty}a_{n,i}\langle
u,Jz_{n,i}\rangle +a_{n,0}\|x_{n}\|^{2}+ \sum_{i=1}^{\infty}a_{n,i} \|z_{n,i}\|^{2}\\
=a_{n,0}\phi(u,x_{n})+\sum_{i=1}^{\infty}a_{n,i}\phi(u,z_{n,i})\\=
a_{n,0}\phi(u,x_{n})+\sum_{i=1}^{\infty}a_{n,i}\inf_{u\in P_{T_{i}}u}\phi(u,z_{n,i})\\\leq
a_{n,0}\phi(u,x_{n})+\sum_{i=1}^{\infty}a_{n,i}\Phi(P_{T_{i}}u,P_{T_{i}}x_{n})\\\leq
a_{n,0}\phi(u,x_{n})+\sum_{i=1}^{\infty}a_{n,i}\phi(u,x_{n})=
\phi(u,x_{n}),
\end{multline*}
which implies that
 $u\in C_{n+1}.$ Hence
$$ \mathcal {F}=\bigcap_{i=1}^{\infty}F(T_{i})\bigcap EP(F)\subset C_{n},\qquad  \forall n\geq 0.$$
We observe that $C_{n}$ is closed and convex (see \cite{mt1,mt2}).
From $x_{n}=\Pi_{C_{n}}x_{0}$, we have \begin{equation}\label{18}\langle
x_{n}-z,Jx_{0}-Jx_{n}\rangle\geq 0, \qquad \forall z\in
C_{n}.\end{equation}
Since $\mathcal{F}\subset C_{n}$ for all $n\geq 0$, we obtain that
$$\langle
x_{n}-u,Jx_{0}-Jx_{n}\rangle\geq 0 \qquad \forall u\in\mathcal{F}.$$
 From Lemma \ref{3}  we have
$$\phi(x_{n},x_{0})=\phi(\Pi_{C_{n}}x_{0},x_{0})\leq
\phi(u,x_{0})-\phi(u,\Pi_{C_{n}}x_{0})\leq \phi(u,x_{0})$$ for all $u\in
\mathcal{F}\subset C_{n} $. Then the sequence $\phi(x_{n},x_{0})$ is
bounded. Thus $\{x_{n}\}$ is bounded. From
$x_{n}=\Pi_{C_{n}}x_{0}$ and  $x_{n+1}\in C_{n+1}\subset C_{n}$ we
have
$$\phi(x_{n},x_{0})\leq\phi(x_{n+1},x_{0}),\qquad \forall n\geq0.$$
Therefore $\{\phi(x_{n},x_{0})\}$ is nondecreasing. So the limit of
$\{\phi(x_{n},x_{0})\}$ exists. By the construction of $C_{n}$ for
any positive integer $m\geq n$ we have $$x_{m}=\Pi _{C_{m}}x_{0}\in
C_{m}\subset C_{n}.$$ It follows that
\begin{multline*}\phi(x_{m},x_{n})=\phi(x_{m},\Pi_{C_{n}}x_{0})\\\leq
\phi(x_{m},x_{0})-\phi(\Pi_{C_{n}}x_{0},x_{0})\\=\phi(x_{m},x_{0})-\phi(x_{n},x_{0}).\end{multline*}
Letting $m,n\to\infty$  we have
\begin{equation}\label{11}\lim_{n\to\infty}\phi(x_{m},x_{n})=0.\end{equation} It follows from Lemma
\ref{1} that $x_{m}-x_{n}\to 0$ as $m,n\to \infty$. Hence $\{x_{n}\}$
is a Cauchy sequence. Since $C$ is closed and convex subset of
Banach space $E$, we can assume that $x_{n}\to z$ as $n\to\infty$.
Next we show $z\in \bigcap_{i=1}^{\infty} F(T_{i})$. By taking $m=n+1$
in (\ref{11}) we get
\begin{equation}\label{20}\lim_{n\to\infty}\phi(x_{n+1},x_{n})=0.\end{equation} It follows
from Lemma \ref{1} that
\begin{equation}\label{12}\lim_{n\to\infty}\|x_{n+1}-x_{n}\|= 0.\end{equation}
From $x_{n+1}=\Pi_{C_{n+1}}x\in C_{n+1}$, we have
$$\phi(x_{n+1},u_{n})\leq \phi(x_{n+1},x_{n}),\qquad n\geq 0,$$
it follows from  (\ref{20}) that
$$\lim_{n\to\infty}\phi(x_{n+1},u_{n})=0.$$
By Lemma \ref{1} we have
\begin{equation}\label{13}\lim_{n\to\infty}\|x_{n+1}-u_{n}\|=0.\end{equation} Combining (\ref{12}) with
(\ref{13}) one see
\begin{equation}\label{14}\lim_{n\to\infty}\|x_{n}-u_{n}\|\leq \lim_{n\to\infty}(\|x_{n+1}-x_{n}\|+\|x_{n+1}-u_{n}\|)=0.\end{equation}
It follows from $x_{n}\to z$ as $n\to\infty$ that $u_{n}\to z$ as
$n\to\infty.$
 Since $J$ is uniformly
norm-to-norm continuous on bounded sets and
$\lim_{n\to\infty}\|x_{n}-u_{n}\|=0$, we have
\begin{equation}\label{15}\lim_{n\to\infty}\|Jx_{n}-Ju_{n}\|=0.\end{equation}
We show that $\{z_{n,i}\}$ is bounded for $i\in\mathbb{N}.$ Indeed,
for $u\in \mathcal{F}$ we have
$$(\|z_{n,i}\|-\|u\|)^{2}\leq \phi(z_{n,i},u)\leq
\phi(x_{n},u)\leq (\|x_{n}\|+\|u\|)^{2}.$$ Since $\{x_{n}\}$ is
bounded, we obtain $\{z_{n,i}\}$ is bounded for $i\in\mathbb{N}.$ Let
$$r=sup_{n\geq 0}\{\|x_{n}\|,\|z_{n,i}\|:i\in\mathbb{N}\}.$$
Since $E$ is a uniformly smooth Banach space, we know that
$E^{*}$ is a uniformly convex Banach space. Therefore from Lemma
\ref{4} there exists a continuous strictly increasing, and convex
function $g$ with $g(0)=0$ such that
\begin{multline}
\phi(u,u_{n})=\phi(u,T_{r_{n}}y_{n})\leq \phi(u,y_{n})\\
=\phi(u,J^{-1}(a_{n,0}Jx_{n}+\sum_{i=1}^{\infty}a_{n,i}Jz_{n,i}))\\
=\|u\|^{2}-2\langle
u,a_{n,0}Jx_{n}+\sum_{i=1}^{\infty}a_{n,i}Jz_{n,i}\rangle+
\|a_{n,0}Jx_{n}+\sum_{i=1}^{\infty}a_{n,i}Jz_{n,i}\|^{2}\\
\leq \|u\|^{2}-2a_{n,0}\langle u,Jx_{n}\rangle -2\sum_{i=1}^{\infty}a_{n,i}\langle
u,Jz_{n,i}\rangle +a_{n,0}\|x_{n}\|^{2}+ \sum_{i=1}^{\infty}a_{n,i} \|z_{n,i}\|^{2}-a_{n,0}a_{n,i} g(\|Jx_{n}-Jz_{n,i}\|)\\
=a_{n,0}\phi(u,x_{n})+\sum_{i=1}^{\infty}a_{n,i}\phi(u,z_{n,i})-a_{n,0}a_{n,i} g(\|Jx_{n}-Jz_{n,i}\|)\\\leq
a_{n,0}\phi(u,x_{n})+\sum_{i=1}^{\infty}a_{n,i}\Phi(P_{T_{i}}u,P_{T_{i}}x_{n})-a_{n,0}a_{n,i} g(\|Jx_{n}-Jz_{n,i}\|)\\\leq
a_{n,0}\phi(u,x_{n})+\sum_{i=1}^{\infty}a_{n,i}\phi(u,x_{n})-a_{n,0}a_{n,i} g(\|Jx_{n}-Jz_{n,i}\|)\\\leq
\phi(u,x_{n})-a_{n,0}a_{n,i} g(\|Jx_{n}-Jz_{n,i}\|).
\end{multline}
It follow that \begin{equation}\label{17}a_{n,0}a_{n,i}
g(\|Jx_{n}-Jz_{n,i}\|)\leq \phi(u,x_{n})-\phi(u,u_{n}) \qquad n
\geq 0.\end{equation} On the other hand
 \begin{multline*}
\phi(u,x_{n})-\phi(u,u_{n})=\|x_{n}\|^{2}-\|u_{n}\|^{2}-2\langle u,Jx_{n}-Ju_{n}\rangle\\
\leq |\,\|x_{n}\|^{2}-\|u_{n}\|^{2}\,|+2|\langle
u,Jx_{n}-Ju_{n}\rangle\,|\\\leq
|\,\|x_{n}\|-\|u_{n}\|\,|(\|x_{n}\|+\|u_{n}\|)+2\|u\|\|
 Jx_{n}-Ju_{n}\|\\\leq
\|x_{n}-u_{n}\|(\|x_{n}\|+\|u_{n}\|)+2\|u\|\|
 Jx_{n}-Ju_{n}\|.
\end{multline*}
It follows from (\ref{14}) and (\ref{15}) that
 \begin{equation}\lim_{n\to\infty}(\phi(u,x_{n})-\phi(u,u_{n}))=0.\end{equation}
Using (\ref{17}) and by assumption that $\liminf a_{n,0}a_{n,i}>0$
we have that
 $$\lim_{n\to\infty}g(\|Jx_{n}-Jz_{n,i}\|)=0,\quad (i\in\mathbb{N}).$$ Therefore from the
property of $g$ , we have
$$\lim_{n\to\infty}\|Jx_{n}-Jz_{n,i}\|=0,\quad (i\in\mathbb{N}).$$
Since $J^{-1}$ is uniformly norm-to-norm continuous on bounded
set, we have $$\lim_{n\to\infty}\|x_{n}-z_{n,i}\|=0,$$ this
implies that
$$\lim_{n\to\infty}dist(x_{n},P_{T_{i}}x_{n})\leq\lim_{n\to\infty}\|x_{n}-z_{n,i}\|=0,\quad
(i\in\mathbb{N}).$$ Now by closedness of $P_{T_{i}}$ we obtain that $z\in
\bigcap_{i=1}^{\infty} F(T_{i})$. By similar argument as in \cite{mt1} (see also\cite{mt2})
we obtain that $z\in EP(F)$.
 Therefore $z\in\mathcal{F}$. Finally we prove
$z=\Pi_{\mathcal{F}}x_{0}$.
 By taking limit in (\ref{18}) we have
$$\langle
z-u,Jx_{0}-Jz\rangle\geq 0,\qquad \forall u\in\mathcal{F}.$$ Hence
by Lemma \ref{2} we have $z=\Pi_{\mathcal{F}}x_{0},$ which complete the proof.
\end{proof}

By similar argument as in the proof of Theorem 3.1, we can prove the following theorem.
\begin{thm}
Let $E$ be a uniformly smooth and uniformly convex Banach space,
and let $C$ be a nonempty closed convex subset of $E$. Let $F$ be
a bifunction from $C\times C$ into $\mathbb{R}$ satisfying
$(A1)-(A4).$ Let $T_{i}:C\to N(C)$,  be a sequence of closed relatively quasi-nonexpansive multivalued
mappings such that  $\mathcal {F}=\bigcap_{i=1}^{\infty}F(T_{i})\bigcap
EP(F)\neq\emptyset$ and for all $p\in\mathcal{F}$, $T_{i}(p)=\{p\}$.  For $x_{0}\in C$ and $C_{0}=C$, let
$\{x_{n}\}$ be a sequence generated by  the following algorithm:
$$\begin {cases}
y_{n}=J^{-1}(a_{n,0}Jx_{n}+\sum_{i=1}^{\infty}a_{n,i}Jz_{n,i}),
\\u_{n}\in C :
F(u_{n},y)+\frac{1}{r_{n}}\langle
y-u_{n},Ju_{n}-Jy_{n}\rangle\geq0 ,\qquad \forall y\in C,\\
C_{n+1}=\{z\in C_{n}: \phi(z,u_{n})\leq \phi(z,x_{n})\},\\
x_{n+1}=\prod_{C_{n+1}}x, \qquad \forall n\geq 0,\end{cases}$$
 where  $\sum_{i=0}^{\infty} a_{n,i}=1$ and  $z_{n,i}\in T_{i}x_{n}$.   Assume further that  $\liminf_{n}a_{n,0}a_{n,i}>0$, $\{r_{n}\}\subset
 (0,\infty)$ and $\liminf_{n}r_{n}>0$.   Then $\{x_{n}\}$ converges strongly
 to $\Pi _{\mathcal{F}}x_{0}$, where $\Pi_{\mathcal{F}}$ is the
 projection of $E$ onto $\mathcal{F}$.
\end{thm}
As a result for single valued mappings we obtain the following
theorem.

\begin{thm}
Let $E$ be a uniformly smooth and uniformly convex Banach space,
and let $C$ be a nonempty closed convex subset of $E$. Let $F$ be
a bifunction from $C\times C$ into $\mathbb{R}$ satisfying
$(A1)-(A4).$ Let Let $T_{i}:C\to C$,  be a sequence of closed relatively quasi-nonexpansive
mappings such that  $\mathcal {F}=\bigcap_{i=1}^{\infty}F(T_{i})\bigcap
EP(F)\neq\emptyset$ .  For $x_{0}\in C$ and $C_{0}=C$, let
$\{x_{n}\}$ be a sequence generated by  the following algorithm:
$$\begin {cases}
y_{n}=J^{-1}(a_{n,0}Jx_{n}+\sum_{i=1}^{\infty}a_{n,i}JT_{i}x_{n}),
\\u_{n}\in C :
F(u_{n},y)+\frac{1}{r_{n}}\langle
y-u_{n},Ju_{n}-Jy_{n}\rangle\geq0 ,\qquad \forall y\in C,\\
C_{n+1}=\{z\in C_{n}: \phi(z,u_{n})\leq \phi(z,x_{n})\},\\
x_{n+1}=\prod_{C_{n+1}}x, \qquad \forall n\geq 0,\end{cases}$$
 where  $\sum_{i=0}^{\infty} a_{n,i}=1$.   Assume further that  $\liminf_{n}a_{n,0}a_{n,i}>0$,
  $\{r_{n}\}\subset (0,\infty)$ and $\liminf_{n}r_{n}>0$. Then $\{x_{n}\}$ converges strongly
 to $\Pi _{\mathcal{F}}x_{0}$, where $\Pi_{\mathcal{F}}$ is the
 projection of $E$ onto $\mathcal{F}$.
\end{thm}

\noindent {\bf Remark }: Our main result generalize the result of Eslamian and Abkar \cite{ea} of a finite family of multivalued mappings to an infinite family of multivalued mappings. We also remove the uniformly continuity of  the mappings.
\section{Some Applications to Hilbert Spaces}
In the Hilbert space setting,  we have
$$\phi(x,y)=\|x-y\|^{2},\qquad \Phi(Tx,Ty)=H(Tx,Ty)^{2}\quad \forall x,y\in H.$$ Therefore
$$\Phi(Tx,Tp)\leq \phi(x,p) \Leftrightarrow H(Tx,Tp)\leq \|x-p\|$$
for every $x\in C$ and $p\in F(T).$ We note that in a Hilbert
space $H$, $J$ is the identity operator.

\begin{thm}
Let $C$ be a nonempty closed convex subset of  a real Hilbert
space $H$. Let $F$ be a bifunction from $C\times C$ into
$\mathbb{R}$ satisfying $(A1)-(A4).$ Let $T_{i}:C\to P(C)$, $
i\in\mathbb{N}$ be a sequence of multivalued mappings such that
$P_{T_{i}}$ is  closed quasi- nonexpansive. Assume that
 $\mathcal {F}=\bigcap_{i=1}^{\infty}F(T_{i})\bigcap
EP(F)\neq\emptyset$. For $x_{0}\in C$ and $C_{0}=C$, let
$\{x_{n}\}$ be a sequences generated by  the following algorithm:
$$\begin {cases}
y_{n}=a_{n,0}x_{n}+\sum_{i=1}^{\infty}a_{n,i}z_{n,i},
\\u_{n}\in C\, such\, that\,
F(u_{n},y)+\frac{1}{r_{n}}\langle
y-u_{n},u_{n}-y_{n}\rangle\geq0 ;\qquad \forall y\in C,\\
C_{n+1}=\{z\in C_{n}: \|z-u_{n}\|\leq \|z-x_{n}\|\}, \\
x_{n+1}=P_{C_{n+1}}x, \qquad \forall n\geq 0\end{cases},
$$
where  $\sum_{i=0}^{\infty} a_{n,i}=1$ and $z_{n,i}\in P_{T_{i}}x_{n}$.  Assume further that  $\liminf_{n}a_{n,0}a_{n,i}>0$
  and $\{r_{n}\}\subset
 [a,\infty)$ for some $a>0$. Then $\{x_{n}\}$ converges strongly
 to $P _{\mathcal{F}}x_{0}$.
\end{thm}

\noindent {\bf Remark }: Theorem 4.1 holds if we assume that $T_{i}$ is closed quasi-nonexpansive  multivalued
mapping and $T_{i}(p)=\{p\}$ for all
$p\in\mathcal{F}$.

\end{document}